\documentclass[12pt]{amsart}
\usepackage{amssymb}
\usepackage[all]{xy}
\usepackage{tikz}

\usepackage{amscd}
\usepackage{amsmath}
\input xy
\xyoption{all}

\usepackage{amsthm}
\usetikzlibrary{arrows}
\usepackage{blindtext}

\usepackage{amsmath}
\usepackage{stmaryrd}
\usepackage{quiver}
\usepackage{amssymb}
\usepackage{mathrsfs}
\usepackage[all]{xy}
\usepackage{tikz}
\usepackage[alphabetic,nobysame,abbrev]{amsrefs}
\usepackage{enumerate}
\RequirePackage{fix-cm}
\usepackage{bm}
\usepackage{mdframed}
\usepackage{cancel}
\usepackage{makecell}
\usepackage{diagbox}

\usepackage{amscd}
\usepackage{amsmath}
\usepackage{bm}
\usepackage{newtxtext} 
\numberwithin{equation}{section}
\usepackage{romanbar}
\usepackage[toc,page]{appendix}
\usepackage{hyperref}
\input xy
\xyoption{all}
\usepackage{color}
\RequirePackage{fix-cm}
\usepackage{bm}
\usepackage{mdframed}
\usepackage{cancel}
\usepackage{makecell}
\usepackage{diagbox}

\DeclareMathAlphabet{\mathpzc}{OT1}{pzc}{m}{it}

\newtheorem{lemma}{Lemma}[section]

\newtheorem{theorem}[lemma]{Theorem}
\newtheorem{proposition}[lemma]{Proposition}

\theoremstyle{definition}
\newtheorem{remark}[lemma]{Remark}
\newtheorem{notation}[lemma]{Notation}
\newtheorem{definition}[lemma]{Definition}

\newtheorem{condition}[lemma]{Condition}

\DeclareMathOperator{\Mod}{Mod}
\DeclareMathOperator{\modd}{mod}

\DeclareMathOperator{\Hom}{Hom}

\DeclareMathOperator{\id}{id}

\DeclareMathOperator{\Ext}{Ext}
\DeclareMathOperator{\Lim}{Lim}

\DeclareMathOperator{\K}{K}
\DeclareMathOperator{\Ker}{Ker}
\DeclareMathOperator{\Coker}{Coker}
\DeclareMathOperator{\Cogen}{Cogen}
\DeclareMathOperator{\cogen}{cogen}

\DeclareMathOperator{\Imm}{Im}
\DeclareMathOperator{\Inj}{Inj}
\DeclareMathOperator{\SII}{SII}

\DeclareMathOperator{\Ann}{Ann}
\DeclareMathOperator{\D}{D}

\DeclareMathOperator{\ts}{t}
\DeclareMathOperator{\fs}{f}

\DeclareMathOperator{\HH}{H}

\DeclareMathOperator{\Prod}{Prod}

\DeclareMathOperator{\Rad}{Rad}

\newtheorem{question}[lemma]{Question}
\newtheorem*{theorem 0*}{Theorem}
\newtheorem*{theorem a*}{Theorem A}
\newtheorem*{theorem b*}{Theorem B}

\newcounter{diagram}
\numberwithin{diagram}{section}

\begin{document}
	
	\title{On simples in a cosilting heart and mutation of cosilting pairs}
	\author{Ramin Ebrahimi} 
\address{School of Mathematical Sciences, Zhejiang Normal University, Jinhua 321004, China}
\email{rebrahimi@zjnu.edu.cn / ramin.ebrahimi1369@gmail.com}
\author{Rasool Hafezi}
\address{School of Mathematics and Statistics, Nanjing University of Information Science and Technology, Nanjing, Jiangsu 210044, P.R. China}
\email{hafezi@nuist.edu.cn}
\author{Jiaqun Wei}
\address{School of Mathematical Science, Zhejiang Normal University, Jinhua 321004, China}
\email{weijiaqun5479@zjnu.edu.cn}

	

	\begin{abstract}
Let $A$ be a finite dimensional algebra. The lattice of torsion pairs in $\modd\text{(}A\text{)}$ is controlled by cosilting pairs, infinitely generated analogues of support $\tau^-$-tilting pairs. Then, edges in the Hasse quiver (i.e. minimal inclusions of torsion-free classes) correspond to irreducible mutation of cosilting pairs. An important difference with classical $\tau$-tilting theory is that not all indecomposable summands of a cosilting pair are mutable.
So, it is very important to identify mutable indecomposable summands in a given cosilting pair.

It is well-known that mutable summands correspond to injective envelopes of finitely presented simples in the HRS-tilted heart. Based on this correspondence, we first present a method for obtaining all simples in the HRS-tilted heart, and then give some necessary and sufficient conditions for an indecomposable summand of a given cosilting pair to be left mutable or right mutable.
	\end{abstract}
	
	\maketitle


\section{Introduction}
Let $A$ be a finite dimensional algebra and $\Mod(A)$ (resp. $\modd(A)$) denotes the category of all (resp. finitely presented) left $A$-modules.
The collection of all torsion pairs in $\modd (A)$, ordered by inclusion of torsion classes forms a complete lattice denoted by $\textbf{tors}(A)$.
In the past decade many authors studied the lattice $\textbf{tors}(A)$ from different perspectives. A major step was done by Adachi, Iyama and Reiten, where they showed that functorially finite torsion pairs are parametrized by $2$-term silting complexes of finitely generated projective $A$-modules, or support $\tau$-tilting pairs. They also showed that minimal inclusions of functorially finite torsion classes correspond to irreducible mutation of associated $2$-term silting complexes, or irreducible mutation of associated support $\tau$-tilting pair. Every support $\tau$-tilting pair has exactly $n$ indecomposable summands, where $n$ is the number of isoclasses of non-isomorphic simple $A$-modules, and we can mutate at each of these $n$ summands.

However, $\tau$-tilting theory just deals with functorially finite torsion pairs.
If we want to parametrize all torsion pairs, we need to work with infinitely generated counterpart of support $\tau$-tilting modules, called silting modules. Then, torsion classes in $\modd (A)$ are in bijection with silting modules, and dually torsion-free classes in $\modd (A)$ are in bijection with cosilting modules. These are not finitely generated in general, so we can't expect for a decomposition into indecomposables. If we work with cosilting modules, because cosilting modules are pure-injective, as a consequence, there is a collection of indecomposables, providing a pseudo-decomposition into indecomposable summands! see Proposition \ref{2.12}.
Indeed, for a cosilting module $C$, if we consider all indecomposable modules that arise as indecomposable summands of products of copies of $C$, and denote this set by $\mathcal{Z}_C$, then $\prod_{X\in\mathcal{Z}_C} X$ is a cosilting module equivalent to $C$.
Let also $\mathcal{I}_C$ be the set of isoclasses of all indecomposable injective $A$-modules with no morphism from $C$ to them. Then the pair $(\mathcal{Z}_C, \mathcal{I}_C)$ is called a cosilting pair, which is an infinitely generated analogous of support $\tau^-$-tilting pairs.
Similar to $\tau$-tilting theory, cosilting pairs (and cosilting modules) are in a natural bijection with $2$-term cosilting complexes in $\K^b(\Inj A)$, by taking minimal injective copresentation.
It was proved in \cite{ALSV} that minimal inclusions of torsion-free classes are corresponds to an operation of mutation at an indecomposable $X$ in the associated cosilting pair (or equivalently, mutation at an indecomposable summand of the corresponding $2$-term cosilting complex). However, in contrary with $\tau$-tilting theory, not all indecomposables in a cosilting pair are mutable. So, the following question is very important.

\begin{question}\label{1.1}
Let $(\mathcal{Z},\mathcal{I})$ be a cosilting pair. In which indecomposables is mutation possible?
\end{question}

In this paper we provide an answer to this question.
Let us first recall that there is another way to detect minimal inclusion of torsion pairs \cite{BKZ,DIRRT}. Let $(\ts,\fs)$ and $(\rm u,v)$ be two torsion pairs in $\modd(A)$. There is a minimal inclusion $\ts\subseteq \rm u$ if and only if there is a torsion-free almost torsion module $B$ for the torsion pair $(\ts,\fs)$ such that $\rm u$ is the smallest torsion class containing $\ts$ and $B$. At the same time, $B$ is a torsion almost torsion-free for $(\rm u,v)$ and $\fs$ is the smallest torsion free class containing $\rm v$ and $B$. This $B$ is a brick and is called the brick label of the minimal inclusion $\ts\subseteq\rm u$, and also the minimal inclusion $\rm v\subseteq \fs$. So, if $(\mathcal{Z},\mathcal{I})$ is the cosilting pair associated to $(\ts,\fs)$, mutable summands are in correspondence with the union of torsion-free almost torsion modules and torsion almost torsion-free modules for the torsion pair $(\ts,\fs)$.

The connection between these two ways of characterizing minimal inclusions of torsion(-free) classes can be revealed through the lens of HRS-tilting, see \cite{ALSV,L}.
Let $t=(\ts,\fs)$ be a torsion pair in $\modd (A)$, and $(\mathcal{T},\mathcal{F})$ be the torsion pair in $\Mod (A)$ obtained by taking direct limit closure of $\ts$ and $\fs$, and $\sigma_t$ be the associated $2$-term cosilting complex. Then, the HRS-tilted t-structure $(\mathcal{X}_t,\mathcal{Y}_t)$ in the unbounded derived category of $\Mod(A)$, $\D(A)$ has interesting properties. The heart of this t-structure, denoted by $\mathcal{H}_t$, is a locally coherent Grothendieck category, and the associated cohomological functor $\HH^0_t$ induces an equivalence between $\Prod(\sigma_t)$ and $\Inj \mathcal{H}_t$, the subcategory of injecive objects in $\mathcal{H}_t$.
It turns out that the operation of mutation at an indecomposable summand $X$ of $\sigma_t$ is possible if and only if $\HH^0_t(X)$ is injective envelope of a finitely presented simple object in $\mathcal{H}_t$. This simple is the brick label mentioned in the previous paragraph. Motivated by this result, first we build a connection between indecomposable injective objects and simple objects in Grothendieck categories.
\begin{theorem}\label{1.2}
Let $\mathcal{H}$ be a Grothendieck category such that its set of isoclasses of indecomposable injective objects, denoted by $\textbf{Sp}\mathcal{H}$, forms a cogenerating subcategory. For any $I\in\textbf{Sp}\mathcal{H}$ consider the canonical morphism
		\[\phi_I:I\longrightarrow \prod_{J\in\textbf{Sp}\mathcal{H}}J^{\Rad(I,J)},\]
where $\Rad(I,J)=\{f\in \Hom_{\mathcal{H}}(I,J)|1_I-gf \space\text{ is invertible for any }\space  g\in\Hom_\mathcal{H}(J,I)\}$. Set $S_I:=\Ker(\phi_I)$ and $\mathcal{S}:=\{S_I\mid I\in\textbf{Sp}\mathcal{H}\; \text{and}\; S_I\neq 0\}$.
Then $\mathcal{S}$ is a complete set of non-isomorphic simple objects in $\mathcal{H}$. And if $S_I\neq 0$, $I$ is the injective envelope of the simple $S_I$.
\end{theorem}

We are interested in the Grothendieck categories that arise as the heart of HRS-tilted t-structure associated to a torsion pair $t=(\ts,\fs)$ in $\modd(A)$. In this case $\HH^0_t(\sigma_t)$ is an injecive cogenerator for $\mathcal{H}_t$.
It is not easy to apply Theorem \ref{1.2} to this $\mathcal{H}_t$, because injective objects are complexes with two non-zero cohomologies.
But, if we first assume that the $2$-term cosilting complex $\sigma_t$, considered as a morphism $\sigma_t:I^0\rightarrow I^1$, is an epimorphism, it would be isomorphic to a stalk complex in $\mathcal{H}_t$.
This is the case exactly when $\sigma_t$ is a $2$-term cotilting complex, or equivalently $C_t:=\Ker(\sigma_t)$ is a cotilting module.
In the end, using results of \cite{ALS2} we can adapt our results to the general case of cosilting torsion pairs.

When $t=(\ts,\fs)$ is a cotilting torsion pair cogenerated by a cotilting module $C$, the associated cohomological functor $\HH^0_t$ sends $\sigma_t$ to $C_t$, considered as a stalk complex concentrated in degree zero.
Then, using Theorem \ref{1.2} we can prove the following.
\begin{theorem}\label{1.3}
Let $C\in \Mod(A)$ be a cotilting module and $\mathcal{Z}_C$ be the set of isoclasses of all indecomposable modules in the product closure of $C$. For any $X\in\mathcal{Z}_C$ consider the canonical morphism
			\[\phi_X:X\longrightarrow \prod_{Y\in\mathcal{Z}_C}Y^{\Rad(X,Y)},\]
		where $\Rad(X,Y)=\{f\in \Hom_{\Lambda}(X, Y)|1_X-gf \space\text{ is invertible for any }\space  g\in\Hom_\Lambda(Y, X)\}$. Set $\mathtt{F}_X:=\Ker(\phi_X)$ and $\mathtt{T}_X:=t(\Coker(\phi_X))$, where $t(-)$ is the torsion functor associated to the torsion pair $({}^{\bot_0}C,\Cogen(C))$ in $\Mod(A)$ cogenerated by $C$.
	Then, exactly one of the following happens.
	\begin{itemize}
	\item[(1)]
	$\mathtt{F}_X=0$ and $\mathtt{T}_X=0$.
	\item[(2)]
	$\mathtt{F}_X\neq 0$ and it is finite dimensional. In this case $\mathtt{T}_X=0$.
	\item[(3)]
	$\mathtt{F}_X\neq 0$ and it is infinite dimensional. In this case $\mathtt{T}_X=0$.
	\item[(4)]
	$\mathtt{T}_X\neq 0$. In this case $\mathtt{T}_X$ is finite dimensional and $\mathtt{F}_X=0$.
	\end{itemize}
	 Furthermore, $\mathbf{F}_C:=\{\mathtt{F}_X\mid X\in\mathcal{Z}_C\; \text{and}\; \mathtt{F}_X\neq 0\}$ is the set of all torsion-free almost torsion modules with respect to $({}^{\bot_0}C,\Cogen(C))$ and $\mathbf{T}_C:=\{T_X\mid X\in\mathcal{Z}_C\; \text{and}\; S_X\neq 0\}$ is the set of all torsion almost torsion-free modules with respect to $({}^{\bot_0}C,\Cogen(C))$. 
\end{theorem} 
Using this theorem, and some results from \cite{ALS2}, we can obtain all torsion-free almost torsion and torsion almost torsion-free modules with respect to  a cosilting torsion pair $(\mathcal{T},\mathcal{F})$. Also we give nessecary and sufficient condition for possiblity of left mutation or right mutation for indecomposable summands in the cosilting pair $(\mathcal{Z}_t,\mathcal{I}_t)$. Indeed, $X$ is right mutable (resp. left mutable) if and only if the situation $(2)$ (resp. $(4)$) in Theorem \ref{1.3} occurs, and $X$ is not mutable if and only if one of the situations $(1)$ or $(3)$ occurs. see Theorem \ref{3.15}.

Also by applying Theorem \ref{1.3}, we can construct the short exact sequences that the authors of \cite{AHL} considered necessary for and indecomposable module in $\mathcal{Z}_C$ to be so-called critical or special neg-isolated, see Theorem \ref{3.10}.
	
	\section{preliminaries}
	\subsection{Notation}
Throughout the paper, $A$ is a finite dimensional algebra. We denote by $\Mod(A)$ (resp. $\modd(A)$) the category of all (resp. finitely presented) right $A$-modules, and by $\tau$, the Auslander-Reiten translation. Also we denote by $\D(A)$ the unbounded derived category of $\Mod(A)$ and by $\K^b(\Inj A)$ (resp. $\K^b(\rm inj A)$) the homotopy category of bounded complexes of (resp. finitely generated) injective $A$-modules.
	
Let $\mathcal{A}$ be an additive category and $\mathcal{X}$ be a class of objects in $\mathcal{A}$. By $\Prod(\mathcal{X})$ (resp. $\rm prod(\mathcal{X})$) we mean the subcategory of all objects in $\mathcal{A}$ that are direct summand of a product (resp. finite product) of objects in $\mathcal{X}$. Moreover, if $\mathcal{A}$ is abelian, we denote by $\Cogen(\mathcal{X})$ (resp. $\cogen(\mathcal{X})$) the subcategory of all objects that are subobject of an object in $\Prod(\mathcal{X})$ (resp. $\rm prod(\mathcal{X})$).
	
	Let $\mathcal{X}$ be a subcategory of $\D(A)$ and $I$ be a subset of integers. We set
	\begin{equation}
		\mathcal{X}^{\bot_I}:=\{M\in \D(A) \mid \Hom_{\D(A)}(X,M[i])=0, \forall i\in I \;\text{and}\; \forall X\in \mathcal{X}\}.\notag
	\end{equation}
${}^{\bot_I}\mathcal{X}$ is defined dually. As usual we denote the interval $\{i\in \mathbb{Z}\mid i\geq 0 \}$ by $\geq 0$ and the interval $\{i\in \mathbb{Z}\mid i\leq 0 \}$ by $\leq 0$. Similar notations will be used for other intervals.
	
In a similar way, for a subcategory $\mathcal{X}\subseteq \Mod(A)$ and a subset of non-negative integers $I$ we define $\mathcal{X}^{\bot_I}$ and $^{\bot_I}\mathcal{X}$.
    
When $I$ consists of only one number, $I=\{n\}$, we will use the notations $\mathcal{X}^{\bot_n}$ and ${}^{\bot_n}\mathcal{X}$ instead of $\mathcal{X}^{\bot_{\{n\}}}$ and ${}^{\bot_{\{n\}}}\mathcal{X}$. Similarly, when $\mathcal{X}=\{X\}$ we use the notations $X^{\bot_I}$ and ${}^{\bot_I}X$ instead of $\{X\}^{\bot_I}$ and $^{\bot_I}\{X\}$.
	
\subsection{Torsion pairs}
Our main objects of study in this paper are torsion pairs in $\modd(A)$. 
But these are closely related to a certain of torsion pairs in $\Mod(A)$ and also some torsion pairs in other abelian categories. So we start from the definition of torsion pair in general abelian categories.
	
\begin{definition}
Let $\mathcal{A}$ be an abelian category. A pair of full subcategories $(\mathcal{T},\mathcal{F})$ is called a {\em torsion pair} if it satisfies the following conditions.
	\begin{itemize}
			\item[(1)] $\Hom_{\mathcal{A}}(T,F)=0$ for every objects $T\in \mathcal{T}$ and $F\in \mathcal{F}$.
			\item[(2)] For any object $X\in \mathcal{A}$, there is a short exact sequence 
		\begin{equation}
				0\rightarrow t(X)\rightarrow X\rightarrow f(X)\rightarrow 0 \notag
		\end{equation}
			with $t(X)\in \mathcal{T}$ and $f(X)\in \mathcal{F}$. This short exact sequence is called the {\it canonical short exact sequence} for $X$.
	\end{itemize}
	In this case $\mathcal{T}$ is called the {\em torsion class} and $\mathcal{F}$ is called the {\em torsion free class}.
\end{definition}
	
Torsion pairs in $\modd(A)$, ordered by inclusion of their torsion classes, form a complete lattice denoted by $\textbf{tors}(A)$ \cite{DIRRT}. Because this lattice encodes essential information about $A$, many authors studied this lattice \cite{AIR,DIJ,BKZ,DIRRT}.
Let us first recall the following definition.
	
\begin{definition}
	Let $\mathcal{A}$ be an additive category and $\mathcal{B}$ be a full subcategory of $\mathcal{A}$. 
	\begin{itemize}
	\item[(1)] 	$\mathcal{B}$ is called a {\it preenveloping subcategory} of $\mathcal{A}$ (or a {\it covariantly finite subcategory} of $\mathcal{A}$), if for every $A\in \mathcal{A}$ there exist an object $B\in\mathcal{B}$ and a morphism $f : A\rightarrow B$ such that, for all $B'\in\mathcal{B}$, the sequence of abelian groups $\Hom_\mathcal{A}(B, B')\rightarrow \Hom_\mathcal{A}(A, B')\rightarrow 0$ is exact. Such a morphism $f$ is called a {\it $\mathcal{B}$-preenvelope of $A$} (or a {\it left $\mathcal{B}$-approximation of $A$}).
	\item[(2)] A $\mathcal{B}$-preenvelope $f: A\rightarrow B$ is called a {\it $\mathcal{B}$-envelope} (or a {\it minimal left $\mathcal{B}$-approximation}), if any morphism $g:B\rightarrow B$ satisfying $gf=f$ is an isomorphism.
	\end{itemize}
	The notions of {\it precovering subcategory} (or {\it contravariantly finite subcategory}), {\it $\mathcal{B}$-precover} (or {\it right $\mathcal{B}$-approximation}) and {\it $\mathcal{B}$-cover} (or {\it minimal right $\mathcal{B}$-approximation}) are defined dually. A {\it functorially finite subcategory of $\mathcal{A}$} is a subcategory which is both covariantly finite and contravariantly finite in $\mathcal{A}$.
	\end{definition}
Given a torsion pair $(\ts,\fs)$ in $\modd(A)$ (or any abelian category), one can immediately see that the canonical short exact sequence for any $M\in\modd(A)$, gives $\ts$-cover and $\fs$-envelop for $M$. Therefore, $\ts$ is covering (contravariantly finite) and $\fs$ is enveloping (covariantly finite).
For torsion pairs in $\modd(A)$ we have the following nice symmetry.
	
\begin{theorem}$($\cite{S}$)$\label{2.3}
Let $(\ts,\fs)$ be a torsion pair in $\modd(A)$. Then $\ts$ is (pre)enveloping if and only if $\fs$ is (pre)covering.
\end{theorem}
A torsion pair $(\ts,\fs)$ in $\modd(A)$ is called {\it functorially finite} if $\ts$, or equivalently $\fs$, is functorially finite.

When the torsion pair $(\ts,\fs)$ is functorially finite, it is possible to paramerize it by a single module. Indeed, if we choose an $\fs$-precover $f:N\rightarrow DA$ for the injective cogenerator $DA$, we can easily see that $\fs=\cogen(N)$. Dually, $\ts=\rm gen(M)$, where $g:A\rightarrow M$ is the $\ts$-preenvelope for $A$ (which is a progenerator).
	
Continuing this line, Adachi, Iyama and Reiten in \cite{AIR} defined support $\tau$-tilting modules and (the dual notion of) support $\tau^-$-tilting modules in order to parametrize functorially finite torsion pairs.
Recall that a pair $(M,I)$ consists of a module $M$ and injecive module $I$ is called a {\it $\tau^-$-tilting pair} if $M$ is $\tau^-$-rigid (which means that $\Hom_A(\tau^-M,M)=0$), $\Hom_A(M,I)=0$ and $|M|+|I|=|A|$. Here $|-|$ counts the number of non-isomorphic indecomposable summands of a given finitely generated module.

For each $\tau^-$-tilting pair $(M,I)$, and a minimal injective copresentation $0\rightarrow M\rightarrow I^0\overset{i}{\rightarrow}I_0$, we can construct the $2$-term complexes
\[I^0\overset{[i,0]^T}{\longrightarrow} I^1\oplus I\]
in $\K^b(\rm inj A)$. This is a $2$-term cosilting complex, see Definition 2.5. And this correspondence gives a bijection between $\tau^-$-tilting pairs and $2$-term cosilting complexes in $\K^b(\rm inj A)$.
	
Let us state one of the main results of \cite{AIR} for $\tau^-$-tilting pairs. The dual statement holds for $\tau$-tilting pairs.
\begin{theorem}\label{2.4}
There is a bijection between
	\begin{itemize}
	\item[(1)] equivalence classes of $\tau^-$-tilting pairs,
	\item[(2)] equivalent classes of $2$-term cosilting complexes in $\K^b(\rm inj A)$, and
	  \item[(3)] functorially finite torsion-free classes in $\modd(A)$.
	\end{itemize}
\end{theorem}
	
	In this correspondence a support $\tau^-$-tilting $A$-module $M$ is sent to the torsion-free class $\cogen(M)$, and conversely a functorially finite torsion-free class $\fs\subseteq \modd(A)$ is send to the direct sum of all pairwise non-isomorphic $\Ext$-injective objects of $\fs$. As we mentioned above, the reason for validity of this result is that for a functorially finite torsion-free class $\fs$ in $\modd(A)$, we can choose an $\fs$-cover $M\rightarrow DA$ and then any module in $\fs$ admits an monomorphism into a finite product of copies of $M$, i.e. $\fs=\rm cogen(M)$. For non-functorially finite torsion-free classes this argument doesn't work.
Below we briefly explain how this problem was solved.
First let recall the following theorem due to Crawley-Boevey \cite{CB}. A torsion pair $(\mathcal{T},\mathcal{F})$ in $\Mod(A)$ is called {\it of finite type}, if $\mathcal{F}$ is closed under direct limits.
	
\begin{theorem} $($\cite{CB}$)$\label{2.5}
There is a bijection between
\begin{itemize}
	\item[(1)]
    torsion pairs in $\modd(A)$, and
	\item[(2)]
    torsion pairs of finite type in $\Mod(A)$.
	\end{itemize}
Under this bijection a torsion pair $(\ts,\fs)$ in $\modd(A)$ is sent to its direct limit closure $(\underrightarrow{\Lim}\ts,\underrightarrow{\Lim}\fs)$.
The inverse map sends a torsion pair $(\mathcal{T},\mathcal{F})$ with $\mathcal{F}=\underrightarrow{\Lim}\mathcal{F}$, to the restricted torsion pair $(\mathcal{T}\cap \modd(A),\mathcal{F}\cap \modd(A))$.
\end{theorem}
	
The good news is that, $\mathcal{F}=\underrightarrow{\Lim}\fs$ is a covering class in $\Mod(A)$ by \cite{El}. So, we can hope that the strategy we used above for functorially finite torsion-free classes will work here too. This was done in \cite{ZW}.
Assume that we are given a torsion pair of finite type $(\mathcal{T},\mathcal{F})$ in $\Mod(A)$. Since $\mathcal{F}$ is covering, we can choose an $\mathcal{F}$-cover $C_0\overset{f}{\rightarrow} DA$ for the injective cogenerator $DA$. By taking kernel we obtain the following exact sequence.
\[0\rightarrow C_1\rightarrow C_0\rightarrow DA.\]
Then $C:=C_0\oplus C_1$ is a cosilting module, in the sense of the following definition, and we have $(\mathcal{T},\mathcal{F})=({}^{\bot_0}C,\Cogen(C))$. Note that this is exactly infinitely generated version of Theorem \ref{2.4} which parametrized functorially finite torsion pairs in $\modd(A)$ in terms of support $\tau^-$-tilting modules.

\begin{definition}\label{2.6}
\begin{itemize}
\item[(1)]
A complex $Z\in \K^b(\Inj A)$ is called a {\it cosilting complex} if
\begin{itemize}
\item[(a)]
$\Hom_{\K^b(\Inj A)}(Z,Z[i])=0$ for all $i\geq 1$.
\item[(b)]
The smallest triangulated subcategory of $\K^b(\Inj A)$ containing $Z$ and closed under products and direct summands is $\K^b(\Inj A)$ itself.
\end{itemize}
\item[(2)]
A cosilting complex $Z$ is called {\it $2$-term}, if $Z^n=0$ for all $n\neq 0,1$. We identify a $2$-term cosilting complex $\cdots\rightarrow Z^0\overset{\sigma}{\rightarrow}Z^1\rightarrow\cdots$ with just the morphism $\sigma$.
\item[(3)]
A module $C\in\Mod(A)$ is called a {\it cosilting module} if there exists an injective copresentation $0\rightarrow C\rightarrow I_0\overset{\sigma}{\rightarrow}I_1$ such that
\[\Cogen(C)=\mathcal{C}_{\sigma}:=\{M\in\Mod(A) \mid \Hom_A(M,\sigma) \text{ is surjective}\}.\]
\end{itemize}
\end{definition}
Every $2$-term cosilting complex $\sigma$ (or any cosilting complex) induces a t-structure $({}^{\perp_{\leq 0}}\sigma,{}^{\perp_{> 0}}\sigma)$ in $\D(A)$.
Two cosilting complexes $\sigma$ and $\sigma'$ are said to be equivalent if they cogenerate the same t-structure or equivalently $\Prod(\sigma)=\Prod(\sigma')$.
On the other hand, in $\Mod(A)$, every cosilting module $C$ cogenerate a torsion pair $({}^{\bot_0}C,\Cogen(C))$, and two cosilting modules $C$ and $C'$ are said to be equivalent if they cogenerate the same torsion pair or equivalently $\Prod(C)=\Prod(C')$. The following proposition simply says that cosilting modules are the shadow of $2$-term cosilting complexes, and the associated torsion pairs are the shadow of corresponding t-structures.

\begin{proposition}\label{2.7}
\begin{itemize}
\item[(1)]
The zeroth cohomology functor $\rm H^0(-):\D(A)\rightarrow \Mod(A)$ induces a bijection between
\begin{itemize}
\item[(a)]
equivalence classes of $2$-term cosilting complexes, and
\item[(b)] 
equivalence classes of cosilting modules.
\end{itemize}
\item[(2)]
$\rm H^0(-)$ sends the t-structure cogenerated by a $2$-term cosilting complex $\sigma$ to the torsion pair cogenerated by the cosilting module $\rm H^0(\sigma)$.
\end{itemize}
\end{proposition}
	
Recall that a {\it cotilting module} (of projective dimension 1) in $\Mod(A)$ is a module $C$ such that $\Cogen(C)={}^{\bot_1}C$. The relation between cosilting modules and cotilting modules is given in the following proposition.
	
\begin{proposition}\label{2.8}
\begin{itemize}
\item[(1)] $($\cite{BP}$)$ A module $C$ is cotilting if and only if it is a cosilting module with respect to an epimorphic injective copresentation.
\item[(2)] $($\cite{ZW}$)$ The following conditions are equivalent for $C\in \Mod(A)$.
		\begin{itemize}
		\item[(a)] $C$ is a cosilting module.
		\item[(b)] $\Cogen(C)$ is a torsion-free class and $C$ is a cotilting module over $\bar{A}=A/\Ann{(C)}$, where $\Ann(C):=\{a\in A\mid aC=0\}$ is the annihilator of $C$.
		\end{itemize}
	\end{itemize}
\end{proposition}
	
\begin{theorem}$($\cite{ZW}$)$\label{2.9}
There is a bijection between the equivalence classes of cosilting $A$-modules and torsion pairs of finite type in $\Mod(A)$. This bijection sends a cosilting module $C$ to the torsion pair $(^{\bot_0}C,\Cogen(C))$.
\end{theorem}
	
Therefore, just as support $\tau^-$-tilting modules and $2$-term cosilting complexes in $\K^b(\rm inj A)$ parametrize functorially finite torsion pairs, cosilting modules and $2$-term cosilting complexes in $\K^b(\Inj A)$ parametrize all torsion pairs in $\modd(A)$. For support $\tau^-$-tilting modules and $2$-term cosilting complexes in $\K^b(\rm inj A)$ we have Krull-Schmidt property, i.e. every such an object has a decomposition to a finite direct (co)product of indecomposable objects. Fortunately, cosilting modules and $2$-term cosilting complexes have a similar property.
The important thing is that cosilting modules and cosilting complexes are pure-injective \cite{MV} as cotilting modules are \cite{B}.
For more information on this, we refer the reader to \cite{ALS1}.

\begin{definition}$($\cite{ALS1}$)$\label{2.10}
\begin{itemize}
\item[(1)]
A set $\mathcal{N}$ of indecomposable pure-injective $2$-term complexes of injective $A$-modules is called {\it rigid} if for every $\mu,\nu\in\mathcal{N}$, $\Hom_{\D(A)}(\mu,\nu[1])=0$. Furtheremor, if it is maximal among all rigid sets, it is called {\it maximal rigid}.
\item[(2)]
A {\it cosilting pair} is a pair $(\mathcal{Z},\mathcal{I})$ given by a set $\mathcal{Z}$ of indecomposable pure-injective $A$-modules and a set $\mathcal{I}$ of indecomposable injectives such that
\begin{itemize}
\item[(i)]
For all $X,Y\in\mathcal{Z}$, all submodules of $X$ are contained in ${}^{\bot_1}Y$.
\item[(ii)]
$\Hom_A(X,I)=0$ for all $X\in\mathcal{Z}$ and $I\in\mathcal{I}$.
\item[(iii)]
Any pair $(\mathcal{Z},\mathcal{I})$ with $(i)$ and $(ii)$ such that $\mathcal{Z}\subseteq\mathcal{Z'}$ and $\mathcal{I}\subseteq\mathcal{I'}$ satisfies  $(\mathcal{Z},\mathcal{I})= (\mathcal{Z}',\mathcal{I}')$.
\end{itemize}
\end{itemize}
\end{definition}

\begin{proposition}$($\cite{ALS1}$)$\label{2.11}
There is a bijection between equivalence classes of $2$-term cosilting complexes and maximal rigid sets, sending a $2$-term cosilting complex $\sigma$ to the set of indecomposable objects in $\Prod(\sigma)$. The inverse map sends a maximal rigid set $\mathcal{N}$ to $\sigma_{\mathcal{N}}:=\prod_{\mu\in\mathcal{N}}\mu$.
\end{proposition}
Again, the above result about the derived category has a shadow in $\Mod(A)$. Indeed, applying the homology functor $\HH^0$ to the above proposition we get:
\begin{proposition}$($\cite{ALS1}$)$\label{2.12}
There is a bijection between equivalence classes of cosilting $A$-modules and cosilting pairs, sending a cosilting module $C$ to $(\mathcal{Z}_C,\mathcal{I}_C)$ where $\mathcal{Z}_C$ is the set of indecomposable modules in $\Prod(C)$ and $\mathcal{I}_C$ is the set of indecomposable injective modules in $C^{\bot_0}$. The inverse map sends a cosilting pair $(\mathcal{Z},\mathcal{I})$ to $\prod_{X\in\mathcal{Z}}X$.
	\end{proposition}
\begin{remark}\label{2.13}
In summary, we have bijections between
\begin{itemize}
\item torsion pairs in $\modd(A)$,
\item torsion pairs of finite type in $\Mod(A)$,
\item cosilting $A$-modules,
\item cosilting pairs,
\item $2$-term cosilting complexes, and
\item maximal rigid sets.
\end{itemize}
\end{remark}
	
Our aim is to study the lattice of torsion pairs $\textbf{tors}(A)$. Having the bijections in Remark \ref{2.13}, we can study the lattice of torsion pairs using other concepts.
	
\subsection{Happel-Reiten-Smalø tilting t-structures}
Any torsion pair $t=(\mathcal{T},\mathcal{F})$ in $\Mod(A)$ induces a t-structure in $\D(A)$, called Happel-Reiten-Smalø tilting t-structure (HRS-tilting for short) at $t$. If the torsion pair is of finite type, the associated t-structure has nice properties. In this subsection, after recalling the definition of t-structure, we recall the connection between HRS-tilting t-structures and $2$-term cosilting complexes. Also, we prove some well-known results specially about the cohomological functor associated to HRS-tilting t-structure, because the cohomological functor is essential in the next section.
	
\begin{definition}
A {\it $t$-structure} in triangulated category $\mathcal{D}$ is a pair of subcategories $(\mathcal{X},\mathcal{Y})$, satisfying the following conditions.
	\begin{itemize}
	\item[(1)] $\Hom_{\mathcal{D}}(X,Y)=0$, for every $X\in\mathcal{X}$ and every $Y\in\mathcal{Y}$.
	\item[(2)] For any object $M\in \mathcal{D}$, there exists a triangle $X_M\rightarrow M\rightarrow Y_M\rightarrow X_M[1]$, with $X_M\in \mathcal{X}$ and $Y_M\in \mathcal{Y}$.
	\item[(3)] $\mathcal{X}[1]\subseteq \mathcal{X}$ and $\mathcal{Y}\subseteq\mathcal{Y}[-1]$.
	\end{itemize}
The subcategory $\mathcal{H}=\mathcal{X}[-1]\cap\mathcal{Y}$ is then an abelian category called the {\it heart} of t-structure $(\mathcal{X},\mathcal{Y})$, and short exact sequences in $\mathcal{H}$ are given by triangles in $\mathcal{D}$ with terms in $\mathcal{H}$. Also, there is a canonical functor $\HH^0:\mathcal{D}\rightarrow \mathcal{H}$ called the {\it cohomological functor} \cite{BBD}.
	\end{definition}
\begin{remark}\label{2.15}
\begin{itemize}
\item[(1)]
It follows from the definition of t-structure that the assignment $M\mapsto X_M$ is functorial. Actually, this defines a functor $\tau^{< 0}:\mathcal{D}\rightarrow \mathcal{H}$, called {\it left truncation}, which is a right adjoint of the inclusion $\mathcal{X}\hookrightarrow \mathcal{D}$.
\item[(2)]
Similarly, the assignment $M\mapsto Y_M$ is functorial and defines a functor $\tau^{\geq 0}:\mathcal{D}\rightarrow \mathcal{H}$, called {\it right truncation}, which is a left adjoint of the inclusion $\mathcal{Y}\hookrightarrow \mathcal{D}$.
\item[(3)]
Putting $\tau^{< n}:=\tau^{< 0}[-n]$ and $\tau^{\geq n}:=\tau^{\geq 0}[-n]$, the cohomological fuctor is defined as $\HH^0=\tau^{< 1}\tau^{\geq 0}\simeq\tau^{\geq 0}\tau^{< 1}$.
\end{itemize}
\end{remark}
For later references, we recall the construction of kernel and cokernel in the heart.

\begin{proposition}\label{2.16}
Let $(\mathcal{X},\mathcal{Y})$ be a t-structure in  triangulated category $\mathcal{D}$ with the heart $\mathcal{H}$ and $a:A\rightarrow B$ be a morphism in $\mathcal{H}$. Extend $a$ to a triangle $A\overset{a}{\rightarrow}B\overset{b}{\rightarrow}C\overset{c}{\rightarrow}A[1]$, and then choose a triangle $X_C\overset{f}{\rightarrow}C\overset{g}{\rightarrow}Y_C\overset{h}{\rightarrow}X[1]$ with $X_C\in\mathcal{X}$ and $Y_C\in\mathcal{Y}$. Then, $\pi:=gb:B\rightarrow Y_C$ is the cokernel of $a$ and $i:=(c f)[-1]$ is the kernel of $a$. The situation is depicted in the following diagram.
\[\begin{tikzcd}
	{X_C[-1]} &&& {X_C} \\
	{C[-1]} & A & B & C \\
	{Y_C[-1]} &&& {Y_C}
	\arrow["{f[-1]}"', from=1-1, to=2-1]
	\arrow["i", shift left, curve={height=-6pt}, squiggly, from=1-1, to=2-2]
	\arrow["f", from=1-4, to=2-4]
	\arrow["{c[-1]}", from=2-1, to=2-2]
	\arrow["{g[-1]}"', from=2-1, to=3-1]
	\arrow["a", from=2-2, to=2-3]
	\arrow["b", from=2-3, to=2-4]
	\arrow["\pi"', shift left, curve={height=-6pt}, squiggly, from=2-3, to=3-4]
	\arrow["g", from=2-4, to=3-4]
\end{tikzcd}\]
\end{proposition}

The standard t-structure $(\D^{< 0},\D^{\geq 0})$ on $\D(A)$ given by
\begin{align}
		&\D^{< 0}=\{X\in \D(A)\mid \HH^i(X)=0, \forall i\geq 0\} \notag \\
		&\D^{\geq 0}=\{X\in \D(A)\mid \HH^i(X)=0, \forall i< 0\}\notag
	\end{align}
recovers $\Mod(A)$ as its heart, up to equivalence. In this case the cohomological functor is the usual zeroth cohomology of complexes.
	
Now assume that we have a torsion pair $t=(\mathcal{T},\mathcal{F})$ in $\Mod(A)$. Happel, Reiten and Smalø \cite{HRS} showed that there exists a $t$-structure $(\mathcal{X}_t,\mathcal{Y}_t)$ in $\D(A)$ given by
	\begin{align}
		&\mathcal{X}_t=\{X\in \D(\Lambda)\mid\HH^0(X)\in \mathcal{T} \space\text{and}\space \HH^i(X)=0, \forall i\geq 1\}, \notag \\
		&\mathcal{Y}_t=\{X\in \D(\Lambda)\mid\HH^0(X)\in \mathcal{F} \space\text{and}\space \HH^i(X)=0, \forall i\leq -1\}.\notag
	\end{align}
	$(\mathcal{X}_t,\mathcal{Y}_t)$ is called {\it HRS-tilt at the torsion pair} $t$. We denote the heart of this $t$-structure by $\mathcal{H}_t$. We can think of $\mathcal{H}_t$ as an abelian category close to $\Mod(A)$. Because $\mathcal{H}_t$ is an important tool for studying properties of the torsion pair $t$, in the following we collect some basic properties of the heart $\mathcal{H}_t$ and the associated cohomological functor. Since the cohomological functor $\HH^0_t$ is obtained by composing left and right truncations, we recall the construction of these truncations in the following lemma. First we fix some notations that we will use in the rest of the paper.
	
	\begin{notation}\label{2.17}
	\begin{itemize}
	\item
	$t=(\ts,\fs)$ is a torsion pair in $\modd(A)$, and by abusing notation we also denote by $t=(\mathcal{T},\mathcal{F})$ the corresponding torsion pair of finite type in $\Mod(A)$ under the bijection of Theorem \ref{2.5}. $t(-)$ denotes the torsion functor associated to the torsion pair $(\mathcal{T},\mathcal{F})$.
	\item
	$\sigma_t$, $C_t$, $\mathcal{N}_t$ and $(\mathcal{Z}_t,\mathcal{I}_t)$ are respectively the $2$-term cosilting complex, the cosilting module, the maximal rigid set and the cosilting pair associated to $t$ as mentioned in Remark \ref{2.13}.
	\item
	The HRS-tilt at $t$ is denoted by $(\mathcal{X}_t,\mathcal{Y}_t)$ with the heart $\mathcal{H}_t$ and the cohomological functor $\HH^0_t$.
	\end{itemize}
	\end{notation}
	
	\begin{lemma}\label{2.18}
	Let
	\[\begin{tikzcd}
	{} & {X:} & \cdots & {X^{-1}} & {X^0} & {X^1} & \cdots
	\arrow[from=1-3, to=1-4]
	\arrow["{d^{-1}}", from=1-4, to=1-5]
	\arrow["{d^0}", from=1-5, to=1-6]
	\arrow["{d^1}", from=1-6, to=1-7]
\end{tikzcd}\]
	be a cochain complex in $\D(A)$, and consider the pull back diagram
	\[\begin{tikzcd}
	&& 0 & 0 & \\
	0 & {\Imm(d^{-1})} & E & {t(\HH^0(X))} & 0 \\
	0 & {\Imm(d^{-1})} & {\Ker(d^0)} & {\HH^0(X)} & 0 \\
	&& {f(\HH^0(X))} & {f(\HH^0(X))} \\
	&& 0 & 0
	\arrow[from=1-3, to=2-3]
	\arrow[from=1-4, to=2-4]
	\arrow[from=2-1, to=2-2]
	\arrow["\alpha", from=2-2, to=2-3]
	\arrow["\id"', from=2-2, to=3-2]
	\arrow[from=2-3, to=2-4]
	\arrow["\beta"', from=2-3, to=3-3]
	\arrow[from=2-4, to=2-5]
	\arrow[from=2-4, to=3-4]
	\arrow[from=3-1, to=3-2]
	\arrow[from=3-2, to=3-3]
	\arrow[from=3-3, to=3-4]
	\arrow[from=3-3, to=4-3]
	\arrow[from=3-4, to=3-5]
	\arrow[from=3-4, to=4-4]
	\arrow["\id"', from=4-3, to=4-4]
	\arrow[from=4-3, to=5-3]
	\arrow[from=4-4, to=5-4]
\end{tikzcd}\]
	where the right-hand column is the canonical short exact sequence with respect to the torsion pair $(\mathcal{T},\mathcal{F})$. Then the following (vertical) short exact sequence of cochain complexes gives us the canonical triangle with respect to the t-structure $(\mathcal{X}_t,\mathcal{Y}_t)$.
	\[\begin{tikzcd}
	\cdots & {X^{-2}} & {X^{-1}} & E & 0 & 0 & \cdots \\
	\cdots & {X^{-2}} & {X^{-1}} & {X^0} & {X^1} & {X^2} & \cdots \\
	\cdots & 0 & 0 & {X^0/E} & {X^1} & {X^2} & \cdots
	\arrow[from=1-1, to=1-2]
	\arrow[from=1-2, to=1-3]
	\arrow["\id"', from=1-2, to=2-2]
	\arrow[from=1-3, to=1-4]
	\arrow["\id"', from=1-3, to=2-3]
	\arrow[from=1-4, to=1-5]
	\arrow[from=1-4, to=2-4]
	\arrow[from=1-5, to=1-6]
	\arrow[from=1-5, to=2-5]
	\arrow[from=1-6, to=1-7]
	\arrow[from=1-6, to=2-6]
	\arrow[from=2-1, to=2-2]
	\arrow[from=2-2, to=2-3]
	\arrow[from=2-2, to=3-2]
	\arrow[from=2-3, to=2-4]
	\arrow[from=2-3, to=3-3]
	\arrow[from=2-4, to=2-5]
	\arrow[from=2-4, to=3-4]
	\arrow[from=2-5, to=2-6]
	\arrow["\id", from=2-5, to=3-5]
	\arrow[from=2-6, to=2-7]
	\arrow["\id", from=2-6, to=3-6]
	\arrow[from=3-1, to=3-2]
	\arrow[from=3-2, to=3-3]
	\arrow[from=3-3, to=3-4]
	\arrow[from=3-4, to=3-5]
	\arrow[from=3-5, to=3-6]
	\arrow[from=3-6, to=3-7]
\end{tikzcd}\]
\begin{proof}
The morphism $X^{-1}\rightarrow E$ is given by the composition $X^{-1}\twoheadrightarrow \Imm(d^{-1})\overset{\alpha}{\rightarrowtail}E$. Thus, the zeroth cohomology of the first row is $t(\HH^0(X))\in\mathcal{T}$ by the above pull back diagram. Similarly, the zeroth cohomology of the third row is $\Ker(d^0)/E\cong f(\HH^0(X))\in\mathcal{F}$. This completes the proof.
\end{proof}
	\end{lemma}
	
\begin{proposition}\label{2.19}
\begin{itemize}
\item[(1)]$($\cite{HRS}$)$
We have
\[\mathcal{H}_t=\{X\in \D(A)\mid \HH^0(X)\in\mathcal{F}, \HH^1(X)\in\mathcal{T}\text{ and } \HH^i(X)=0\text{ }\forall i\neq 0,1\}.\]
In particular, $\mathcal{H}_t$ admits a torsion pair $(\mathcal{F},\mathcal{T}[-1])$.
\item[(2)]$($\cite[Theorem 1.3]{ZW}$)$
$(\mathcal{X}_t,\mathcal{Y}_t)=({}^{\bot_{\leq 0}}\sigma_t,{}^{\bot_{> 0}}\sigma_t)$.
\item[(3)]$($\cite[Theorem 5.2]{Sa}$)$
$\mathcal{H}_t$ is a locally coherent Grothendieck category, and the subcategory of finitely presented objects is given by
\[\mathcal{H}_t^{fp}=\{X\in \D(A)\mid \HH^0(X)\in\fs, \HH^1(X)\in\ts\text{ and } \HH^i(X)=0\text{ }\forall i\neq 0,1\}.\]
\item[(4)]$($\cite[Lemma 2.8]{AMV2}$)$
The cohomological functor $\HH^0_t$ restricts to an equivalence between $\Prod(\sigma_t)$ and $\Inj \mathcal{H}_t$, the subcategory of injective objects of $\mathcal{H}_t$. In particular, $\HH^0_t$ induces a bijection between $\mathcal{N}_{\sigma}$ and the set of isoclasses of indecomposable injective objects in $\mathcal{H}_t$.
\end{itemize}
\end{proposition}
	
	We saw that the heart $\mathcal{H}_t$ is a locally coherent Grothendieck category and its injective objects are the shadows of $2$-term cosilting complexes in $\Prod(\sigma_t)$ after applying the cohomological functor $\HH^0_t$. Beside injective objects, simple objects in abelian categories are another kind of generators.
	Simple objects in $\mathcal{H}_t$ were studied in \cite{AHL}.
	Since $\mathcal{H}_t$ has a torsion pair $(\mathcal{F},\mathcal{T}[-1])$, any simple object either belongs to $\mathcal{F}$ or belongs to $\mathcal{T}[-1]$. Following \cite{AHL} we call simples in $\mathcal{F}$ {\it torsion simple} and simples in $\mathcal{T}[-1]$ {\it torsion-free simple}.
    
\begin{definition}$($\cite[Definition 3.1]{AHL}$)$\label{2.20}
let $(\mathcal{T},\mathcal{F})$ be a torsion pair in abelian category $\mathcal{A}$. $F\in \mathcal{F}$ is called {\it torsion-free almost torsion} if
\begin{itemize}
	\item[(1)] Every proper quotient of $F$ is contained in $\mathcal{T}$.
	\item[(2)] For every $F'\in \mathcal{F}$ and every short exact sequence $0\rightarrow F\rightarrow F'\rightarrow C\rightarrow 0$, $C\in \mathcal{F}$.
\end{itemize}
The concept of {\it torsion almost torsion-free} is defined dually.
	\end{definition}
	
We mention that torsion-free almost torsion (resp. torsion almost torsion-free) module for torsion pairs in $\modd(A)$ were studied by Barnard, Carroll, and Zhu under the name {\it minimal extending modules} (resp. {\it minimal coextending modules}) \cite{BKZ}.
	
\begin{proposition}\label{2.21}
	\begin{itemize}
	\item[(1)]$($\cite{AHL}$)$
     Torsion simples in the heart $\mathcal{H}_t$ are exactly torsion-free almost torsion modules with respect to $(\mathcal{T},\mathcal{F})$.
	\item[(2)]$($\cite{AHL}$)$
    Torsion-free simples in the heart $\mathcal{H}_t$ coincide with objects of the form $T[-1]$, where $T$ is torsion almost torsion-free modules with respect to $(\mathcal{T},\mathcal{F})$.
     \item[(3)]$($\cite{S}$)$
     Torsion almost torsion-free modules are always finite dimensional.
	\end{itemize}
	\end{proposition}
	
	From the above proposition it is obvious that every torsion-free almost torsion and every torsion almost torsion-free module is a brick, i.e. the endomorphism ring is a division ring. 
	
	By Proposition \ref{2.19} the cohomological functor $\HH^0_t$ sends $2$-term cosilting complex $\sigma_t$ to an injective cogenerator for the heart $\mathcal{H}_t$. In the following we explicitly describe $\HH^0_t(\sigma_t)$.
	
	\begin{proposition}\label{2.22}
	Let $\sigma_t=I_0\overset{\sigma}{\rightarrow}I_1$ and $\mathcal{N}_t$ be the associated maximal rigid set.
    \begin{itemize}
    \item[(1)]
    $\HH^0_t(\sigma_t)=I_0\rightarrow W$, where $W$ is obtained by the following pull back diagram.
    \[\begin{tikzcd}
	0 & {\Imm(\sigma)} & W & {t(\Coker(\sigma))} & 0 \\
	0 & {\Imm(\sigma)} & {I_1} & {\Coker(\sigma)} & 0
	\arrow[from=1-1, to=1-2]
	\arrow[from=1-2, to=1-3]
	\arrow["\id"', from=1-2, to=2-2]
	\arrow[from=1-3, to=1-4]
	\arrow[from=1-3, to=2-3]
	\arrow[from=1-4, to=1-5]
	\arrow[from=1-4, to=2-4]
	\arrow[from=2-1, to=2-2]
	\arrow[from=2-2, to=2-3]
	\arrow[from=2-3, to=2-4]
	\arrow[from=2-4, to=2-5]
\end{tikzcd}\]
    \item[(2)]
    Let $\mu:E^0\rightarrow E^1\in\mathcal{N}_t$, then $\HH^0_t(\sigma_t)=E^0\rightarrow W_{\mu}$ where $W_{\mu}$ is obtained by a pull back diagram similar to $(1)$.
    \end{itemize}	
    \begin{proof}
    Since $\Ker(\sigma_t)=C\in \mathcal{F}$, we already have $\sigma_t\in \mathcal{Y}_t$. So by Remark \ref{2.15} $(3)$, we only need to take the left truncation of $\sigma_t$ with respect to the t-structure $(\mathcal{X}_t[-1],\mathcal{Y}_t[-1])$. Then the results follow from Lemma \ref{2.18}.
    \end{proof}
	\end{proposition}

	As we mentioned in Proposition \ref{2.8}, cotilting modules are exactly cosilting modules with respect to an epimorphic injective copresentation. In particular, any cotilting module is isomorphic to the associated $2$-term cosilting complex in $\D(A)$. Thus, the following result is an immediate consequence of Proposition \ref{2.22}.
	
	\begin{proposition}\cite{CGF}\label{2.23}
		Let $C$ be a cotilting module and $t=t_C=(^{\bot_0}C,\Cogen(C))$ be
		the torsion pair cogenerated by $C$. Then $C$ (considered as a stalk complex concentrated in degree zero) is an injective cogenerator for $\mathcal{H}_t$. In particular, indecomposable modules in $\Prod(C)$ coincide with indecomposable injective objects in $\mathcal{H}_t$.
	\end{proposition}
	
\section{The main results}
	In this section, first we give a method to obtain all simple objects in Grothendieck categories satisfying some mild condition. Then, we will apply this result to the heart of HRS-tilt at a torsion pair in $\modd(A)$.
	
	Probably, the most important examples of Grothendieck categories are module categories. In this case, indecomposable injective modules cogenerate all injectives and also cogenerate the whole module category.
	\begin{proposition}\label{3.1}
	Let $R$ be an arbitrary ring with identity. Then, the set of isoclasses of indecomposable injective $R$-modules is a cogenerating subcategory of $\Mod(R)$. In particular, every injective $R$-module is a direct summand of a product of indecomposable injective modules.
	\begin{proof}
	Let $M\in\Mod(R)$. For every $x\in M$ denote by $Rx$ the cyclic module generated by $x$. Because $Rx$ is finitely generated it has a maximal submodule and so a simple quotient $S_x$. Let $E(S_x)$ be the injective envelope of this simple. Then, by injectivity, there exists some morphism $h_x$ making the following diagram commutative.
	\[\begin{tikzcd}
	Rx & M \\
	{S_x} \\
	{E(S_X)}
	\arrow[hook, from=1-1, to=1-2]
	\arrow[two heads, from=1-1, to=2-1]
	\arrow["{h_x}", dashed, from=1-2, to=3-1]
	\arrow[hook, from=2-1, to=3-1]
\end{tikzcd}\]
It is clear that the canonical morphism $M\rightarrow \prod_{x\in M}E(S_x)$ induced by all $h_x$'s is a monomorphism. 
	\end{proof}
	\end{proposition}
	
	There is no guarantee that the condition of the above proposition holds in general Grothendieck categories. However, the categories that are important to us satisfy this condition.
	\begin{condition}$(\SII)$
	We say that a Grothendieck category satsisfies the condition $(\SII)$ (stands for \textbf{S}ufficiency of \textbf{I}ndecomposable \textbf{I}njectives), if the collection of indecomposable injective objects forms a cogenerating subcategory. Or equivalently, every injective object is a direct summand of a product of indecomposable injective objects.
	\end{condition}
The radical ideal of an additive category $\mathcal{A}$ is defined as follow. For a pair of objects $(X,Y)$ in $\mathcal{A}$,
\[\Rad(X,Y):=\{f\in \Hom_{\mathcal{A}}(X,Y)|1_X-gf \space\text{ is invertible for any }\space  g\in\Hom_\mathcal{A}(Y,X)\}.\]

We start with the following easy lemma which seems to be known, at least for module categories.
	\begin{lemma}\label{3.3}
	Let $\mathcal{A}$ be an additive category and $X$ and $Y$ be two indecomposable objects with local endomorphism rings. Then, $\Rad(X,Y)$ is the set of non-isomorphisms $f:X\rightarrow Y$. 
	\begin{proof}
	An isomorphism $f:X\rightarrow Y$ doesn't belong to $\Rad(X,Y)$ because $1_X-f^{-1}f=0$ is not invertible.
	Conversely, assume that $f:X\rightarrow Y$ is a non-isomorphism and $g:Y\rightarrow X$ is an arbitrary morphism. Then, if we knew that $gf$ is not isomorphism, $1_X-gf$ would be an isomorphism, because other wise $1=gf+1_X-gf$ is not isomorphism by locality assumption.
	
	It remains to prove that $gf$ cannot be an isomorphism. Assume that $gf$ is an isomorphism and $\alpha$ be its inverse. Then $f\alpha g:Y\rightarrow Y$ is an idempotent. From $\alpha gf=1_X$ we know that $\alpha g$ is an epimorphism. So, if $f\alpha g=0$, $f=0$ which contradicts the assumption that $gf$ is an isomorphism.  Similarly $f\alpha g=1_Y$ implies that $f$ is an isomorphism, a contradiction. Therefore, $f\alpha g$ is a non-trivial idempotent in the endomorphism ring of $Y$, which contradicts locality.
	\end{proof}
	\end{lemma}
    
Note that indecomposable injective objects in Grothendieck categories have local endomorphism ring \cite[Lemma 2.5.7]{Kr}.
	 
Now we can prove the following main result, which enables us to obtain all simple objects from indecomposable injective objects.

\begin{theorem}\label{3.4}
Let $\mathcal{H}$ be a Grothendieck category which satisfies the condition $(\SII)$, and $\textbf{Sp}\mathcal{H}$ be the set of isoclasses of indecomposable injective objects in $\mathcal{H}$. For each $I\in\textbf{Sp}\mathcal{H}$, consider the following canonical morphism 
			\[\phi_I:I\longrightarrow \prod_{J\in\textbf{Sp}\mathcal{H}}J^{\Rad(I,J)}.\]
Set $S_I:=\Ker(\phi_I)$ and $\mathcal{S}:=\{S_I\mid I\in\textbf{Sp}\mathcal{H}\; \text{and}\; S_I\neq 0\}$.
Then, $\mathcal{S}$ is a complete set of non-isomorphic simple objects in $\mathcal{H}$. Moreover, if $S_I\neq 0$, $I$ is the injective envelope of the simple $S_I$.
\begin{proof}
We prove that if $S_I\neq 0$, then it is a simple object. Assume that $M$ is a non-zero subobject of $S_I$. By taking injecive envelope of $I/M$ and using the condition $(\SII)$, we get the commutative diagram
	\[\begin{tikzcd}
	& M && \\
	{S_I} & I && {\prod_{J\in\textbf{Sp}\mathcal{H}}J^{\Rad(I,J)}} \\
	{} & {I/M} \\
	& {E(I/M)} && {\prod_{J\in\textbf{Sp}\mathcal{H}}J^{\alpha_J}}
	\arrow[hook, from=1-2, to=2-1]
	\arrow[hook, from=1-2, to=2-2]
	\arrow[hook, from=2-1, to=2-2]
	\arrow["{\phi_I}", from=2-2, to=2-4]
	\arrow["\pi"', two heads, from=2-2, to=3-2]
	\arrow["i"', hook, from=3-2, to=4-2]
	\arrow["s", hook, from=4-2, to=4-4]
\end{tikzcd}\]
for some family of sets $(\alpha_J)_{J\in\textbf{Sp}\mathcal{H}}$. Each component of the morphism 
\[si\pi:I\longrightarrow \prod_{J\in\textbf{Sp}\mathcal{H}}J^{\alpha_J}\]
to some $J\in \textbf{Sp}\mathcal{H}$ is a non-isomorphism, because $\Ker(si\pi)=\Ker(\pi)=M\neq 0$. So, by the definition of $S_I$, $S_I$ is a subobject of $\Ker(si\pi)=M$ which means that $M=S_I$.

So, we have proved that if $S_I\neq 0$, it is simple. Then because $I$ is indecomplosable, it is obviously the injective envelope of $S_I$. Now we want to prove that $\mathcal{S}$ is a complete set of non-isomorphic simple objects. Assume that $S$ is a simple object and $I$ be its injective envelope, which is clearly indecomposable. We prove that $S=S_I$. Since $S\subseteq I$ is essential, the kernel of any non-isomorphism $I\rightarrow J$ cantains $S$, so $S\subseteq S_I$. On the other hand, by taking injecive envelope of $I/S$ and using the condition $(\SII)$, similar to the above paragraph, we get the following diagram of morphisms.
\[\begin{tikzcd}
	I & {I/S} & {E(I/S)} & {\prod_{J\in\textbf{Sp}\mathcal{H}}J^{\alpha_J}}
	\arrow["\pi", from=1-1, to=1-2]
	\arrow["i", hook, from=1-2, to=1-3]
	\arrow["s", hook, from=1-3, to=1-4]
\end{tikzcd}\]
Similar to the above argument, each component of this morphism is a non-isomorphism. Thus $S_I\subseteq\Ker(si\pi)=S$. This completes the proof.
	\end{proof}
	\end{theorem}
	
After this general result, we now get back to our set up. For a given torsion pair $t=(\ts,\fs)$ in $\modd(A)$ we are interested in HRS-tilted heart $\mathcal{H}_t$, see Notation \ref{2.17}.
\begin{lemma}\label{3.5}
The Grothendieck category $\mathcal{H}_t$ satisfies the condition $(\SII)$.
\begin{proof}
First note that because $\Prod(\sigma_t)$ (resp. $\Inj\mathcal{H}_t$) is a product-closed subcategory of $\D(A)$ (resp. $\mathcal{H}_t$), the restriction of the cohomological functor $\HH^0_t:\D(A)\rightarrow \mathcal{H}_t$ to
\[\HH^0_t|_{\Prod(\sigma_t)}:\Prod(\sigma_t)\rightarrow \Inj\mathcal{H}_t\]
which is an equivalence by Proposition \ref{2.19},
both preserves and reflects products. Also, $\HH^0_t$ induces a bijection between the maximal rigid set $\mathcal{N}_t$ and $\textbf{Sp}\mathcal{H}_t$.
	 
Now, because every object in $\Prod(\sigma_t)$ is a direct summand of a product of elements of $\mathcal{N}_t$ (which are indecomposable) by Proposition \ref{2.11}, every object in $\Inj\mathcal{H}_t$ is a direct summand of a product of elements of $\textbf{Sp}\mathcal{H}_t$ (which are indecomposable).
	\end{proof}
	\end{lemma}
By the above lemma we can apply Theorem \ref{3.4} to $\mathcal{H}_t$.
For simplicity, first we assume that $\sigma_t$ is a $2$-term cotilting complex, i.e. $\sigma_t$ is an epimorphism. In this case, $\sigma_t$ is isomorphic to the cotilting module $C_t=\HH^0(\sigma_t)$ in $\D(A)$. Thus, we can work with $C_t$.
By Proposition \ref{2.23} $\Prod(C_t)$, considered as a subcategory of $\D(A)$ concentrated in degree zero, is the subcategory of injective objects in $\mathcal{H}_t$. In particular, $\textbf{Sp}\mathcal{H}_t$ coincides with the set of indecomposable objects in $\Prod(C_t)$.
	
Now we can prove the following main result.

\begin{theorem}\label{3.6}
Let $C\in \Mod(A)$ be a cotilting module and $\mathcal{Z}_C$ be the set of isoclasses of all indecomposable modules in $\Prod(C)$. For each $X\in\mathcal{Z}_C$ consider the canonical morphism.
\begin{equation}
\phi_X:X\longrightarrow \prod_{Y\in\mathcal{Z}_C}Y^{\Rad(X,Y)}
\end{equation}
Set $\mathtt{F}_X:=\Ker(\phi_X)$ and $\mathtt{T}_X:=t(\Coker(\phi_X))$, where $t(-)$ is the torsion functor associated to the torsion pair $t=(\mathcal{T},\mathcal{F}):=({}^{\bot_0}C,\Cogen(C))$.
Then, exactly one of the following happens.
	\begin{itemize}
	\item[(1)]
    $\mathtt{F}_X=0$ and $\mathtt{T}_X=0$.
	\item[(2)]
    $\mathtt{F}_X\neq 0$ and $\mathtt{T}_X=0$.
	\item[(3)]
	$\mathtt{F}_X=0$ and $\mathtt{T}_X\neq 0$.
	\end{itemize}
Furthermore, $\mathbf{F}_C:=\{\mathtt{F}_X\mid X\in\mathcal{Z}_C\; \text{and}\; \mathtt{F}_X\neq 0\}$ is the set of all torsion-free almost torsion modules with respect to $(\mathcal{T},\mathcal{F})$ and $\mathbf{T}_C:=\{T_X\mid X\in\mathcal{Z}_C\; \text{and}\; S_X\neq 0\}$ is the set of all torsion almost torsion-free modules with respect to $(\mathcal{T},\mathcal{F})$. 
	 \begin{proof}
If we consider $\phi_X$ as a morphism in $\mathcal{H}_t$, it fits in the following triangle in $\D(A)$, where we have denoted $\prod_{Y\in\mathcal{Z}_C}Y^{\Rad(X,Y)}$ by $\textbf{Y}$.
	\[\begin{tikzcd}
	\vdots & \vdots & \vdots & \vdots \\
	0 & 0 & 0 & X \\
	X & X & {\textbf{Y}} & {\textbf{Y}} \\
	{\textbf{Y}} & 0 & 0 & 0 \\
	\vdots & \vdots & \vdots & \vdots
	\arrow[from=1-1, to=2-1]
	\arrow[from=1-2, to=2-2]
	\arrow[from=1-3, to=2-3]
	\arrow[from=1-4, to=2-4]
	\arrow[from=2-1, to=2-2]
	\arrow[from=2-1, to=3-1]
	\arrow[from=2-2, to=2-3]
	\arrow[from=2-2, to=3-2]
	\arrow[from=2-3, to=2-4]
	\arrow[from=2-3, to=3-3]
	\arrow["{\phi_X}", from=2-4, to=3-4]
	\arrow["\id", from=3-1, to=3-2]
	\arrow["{\phi_X}"', from=3-1, to=4-1]
	\arrow["{\phi_X}", from=3-2, to=3-3]
	\arrow[from=3-2, to=4-2]
	\arrow["\id", from=3-3, to=3-4]
	\arrow[from=3-3, to=4-3]
	\arrow[from=3-4, to=4-4]
	\arrow[from=4-1, to=4-2]
	\arrow[from=4-1, to=5-1]
	\arrow[from=4-2, to=4-3]
	\arrow[from=4-2, to=5-2]
	\arrow[from=4-3, to=4-4]
	\arrow[from=4-3, to=5-3]
	\arrow[from=4-4, to=5-4]
\end{tikzcd}\]
Similar to the proof of Proposition \ref{2.22},  the kernel of $\phi_X$ in $\mathcal{H}_t$ is the $2$-term complex $K_X:=X\rightarrow W$, which is obtained by the following pull back diagram.
\[\begin{tikzcd}
	0 & {\Imm(\phi_X)} & W & {t(\Coker(\phi_X))} & 0 \\
	0 & {\Imm(\phi_X)} & {\textbf{Y}} & {\Coker(\phi_X)} & 0
	\arrow[from=1-1, to=1-2]
	\arrow[from=1-2, to=1-3]
	\arrow["\id"', from=1-2, to=2-2]
	\arrow[from=1-3, to=1-4]
	\arrow[from=1-3, to=2-3]
	\arrow[from=1-4, to=1-5]
	\arrow[from=1-4, to=2-4]
	\arrow[from=2-1, to=2-2]
	\arrow[from=2-2, to=2-3]
	\arrow[from=2-3, to=2-4]
	\arrow[from=2-4, to=2-5]
\end{tikzcd}\]
So, $\HH^0(K_X)=\Ker(\phi_X)$ and $\HH^1(K_X)=t(\Coker(\phi_X))$.
By Theorem \ref{3.4} either $K_X=0$ or it is a simple object in $\mathcal{H}_t$. Furthermore, because $\mathcal{H}_t$ admits a torsion pair $(\Cogen(C),{}^{\bot_0}C[-1])$ by Proposition \ref{2.12}, if $K_X\neq 0$ it most be either in $\Cogen(C)$ or in ${}^{\bot_0}C[-1]$. Thus, exactly one of the following happens.
\begin{itemize}
    \item[(1)]
    $K_X=0$, which means that $K_X$ is an acyclic complex, i.e. $\mathtt{F}_X=0$ and $\mathtt{T}_X=0$.
    \item[(2)]
    $K_X\in \mathcal{F}$, which means that $\mathtt{F}_X\neq 0$ and $\mathtt{T}_X=0$.
	\item[(3)]
    $K_X\in \mathtt{T}[-1]$, which means that
	$\mathtt{F}_X=0$ and $\mathtt{T}_X\neq 0$.
	\end{itemize}
The other statements follow from Theorem \ref{3.4} and Proposition \ref{2.21}.  
	 \end{proof}
\end{theorem} 

\begin{proposition}\label{3.7}
Keeping the notations of Theorem \ref{3.6},
\begin{itemize}
\item[(1)] If $\mathtt{F}_X\neq 0$, then $X$ is the injective envelope of the torsion simple $\mathtt{F}_X$.
\item[(2)] If $\mathtt{T}_X\neq 0$, it is finite dimensional and $X$ is the injective envelope of the torsion-free simple $\mathtt{T}_X[-1]$.
\end{itemize}
\begin{proof}
If $\mathtt{T}_X\neq 0$, then it is a torsion almost torsion-free with respect to torsion pair $({}^{\bot_0}C,\Cogen(C))$, and so by Proposition \ref{2.21} it is finite dimensional. Other statements are direct consequences of Theorems \ref{3.4} and \ref{3.6}.
\end{proof}
\end{proposition}

The authors of \cite{AHL} studied simple objects in the heart associated to a cotilting module, and injective envelopes of these simples.
They proved that, for a given indecomposable module in the product closure of a cotilting module, being the injective envelope of simple objects in the heart is equivalent to existence of certain short exact sequences, see Theorem \ref{3.9}.
In the sequel, after recalling this result and some necessary terminologies, we provide a different way to detect injective envelopes of simples in the heart using Theorem \ref{3.6}.
We also explicitly construct the short exact sequences that the authors of \cite{AHL} considered necessary for being injective envelope of simples in the heart.

Recall that a {\it left almost split morphism} $f:M\rightarrow N$ in an additive category is a morphism that is not split monomorphism, and any other morphism $f':M\rightarrow N'$ which is not split monomorphism factor through $f$. When such a factorization is unique, $f$ is called {\it strong left almost split}. The notions of {\it right almost split morphism} and {\it strong right almost split morphism} is defined dually.
\begin{definition}$($\cite{AHL}$)$
Let $t=(\ts,\fs)$ be a torsion pair in $\modd(A)$.
\begin{itemize}
\item[(1)]
$X\in \mathcal{Z}_t$ is called {\it neg-isolated} if there exists a left almost split morphism $X\rightarrow X^{+}$ in $\mathcal{F}$.
\item[(2)]
$X\in \mathcal{Z}_t$ is called {\it critical neg-isolated}, or just {\it critical}, if there exists a left almost split morphism $X\rightarrow X^{+}$ in $\mathcal{F}$ which is an epimorphism. 
\item[(3)]
$X\in \mathcal{Z}_t$ is called {\it special neg-isolated}, or just {\it special}, if there exists a left almost split morphism $X\rightarrow X^{+}$ in $\mathcal{F}$ which is an monomorphism. 
\end{itemize}
\end{definition}

Assume that $X\in \mathcal{Z}_t$ is neg-isolated and $f:X\rightarrow X^{+}$ be a left almost split morphism in $\mathcal{F}$. If $f$ is a monomorphism, $X$ is special by definition. Otherwise, $X\rightarrow\Imm(f)$ is a left almost split morphism in $\Cogen(C)$ which is an epimorphism. Thus, neg-isolated modules are either critical or special.
We want to use Theorem \ref{3.6} to identify critical and special modules in $\mathcal{Z}_t$. First, we recall the following important result, which clarifies the relation between critical modules and torsion-free almost torsion modules, and also the relation between special modules and torsion almost torsion-free modules.

\begin{theorem}$($\cite{AHL}$)$\label{3.9}
Let $t=(\mathcal{T},\mathcal{F}):=({}^{\bot_0}C,\Cogen(C))$ be the torsion pair $\Mod(A)$ cogenerated by the cotilting module $C$, and $\mathcal{Z}_C$ be the set of indecomposable modules in $\Prod(C)$.
\begin{itemize}
	\item[(1)] Consider the short exact sequence
	\begin{equation}\label{eq3.2}
	0\rightarrow F\overset{g}{\longrightarrow} M\overset{a}{\longrightarrow} \bar{M}\rightarrow 0
	\end{equation}
    in $\Mod(A)$. We have that $F$ is torsion-free almost torsion with respect to $t$ and $g$ is a $\Prod(C)$-envelope if and only if $M$ is critical and $a$ is a left almost split map in $\mathcal{F}$. In particular, $M\in\mathcal{Z}_C$ is critical if and only if it is injective envelope of a torsion simple in the heart $\mathcal{H}_t$. 
			\item[(2)] Consider the short exact sequence
	\begin{equation}\label{eq3.3}
	0\rightarrow N\overset{b}{\longrightarrow} \bar{N}\overset{f}{\longrightarrow} T\rightarrow 0
	\end{equation}
	in $\Mod(A)$. We have that $T$ is torsion almost torsion-free with respect to $t$ and $f$ is an $\mathcal{F}$-cover if and only if $N$ is special and $b$ is a left almost split map in $\mathcal{F}$. In particular, $M\in\mathcal{Z}_C$ is special if and only if it is injective envelope of a torsion-free simple in the heart $\mathcal{H}_t$. 
		\end{itemize}
	\end{theorem}
	
	Now we can prove the following theorem, which determines whether an indecomposable summand $X\in\mathcal{Z}_C$ is critical, special, or neither.
	Here kernels and cokernels are computed in $\Mod(A)$. $(2)$ was proved in \cite{EN}.
	\begin{theorem}\label{3.10}
	Let $C\in \Mod(A)$ be a cotilting module. For any $X\in\mathcal{Z}_C$ consider the canonical morphism
	\begin{equation*}
		\phi_X:X\longrightarrow \prod_{Y\in\mathcal{Z}_C}Y^{\Rad(X,Y)},
	\end{equation*}
	constructed in Theorem \ref{3.6}.
	\begin{itemize}
	\item[(1)]
	$X$ is not neg-isolated if and only if $\Ker(\phi_X)=0$ and $\Coker(\phi_X)\in\mathcal{F}$.
	\item[(2)]
	$X$ is critical if and only if $\Ker(\phi_X)\neq 0$. In this case, the induced short exact sequence
	\[0\rightarrow\Ker(\phi_X)\overset{g}{\longrightarrow} X\overset{a}{\longrightarrow} \Imm(\phi_X)\rightarrow 0\]
	satisfies the requirements of Theorem \ref{3.9} for the short exact sequence \eqref{eq3.2}.
	\item[(3)]
	$X$ is special if and only if $\Coker(\phi_X)\notin\mathcal{F}$. In this case, $\phi_X$ is a monomorphism and the upper row of the following pull back diagram satisfies the requirements of Theorem \ref{3.9} for the short exact sequence \eqref{eq3.3}, where $\textbf{Y}:=\prod_{Y\in\mathcal{Z}_C}Y^{\Rad(X,Y)}$.
	\[\begin{tikzcd}
	0 & X & {\bar{X}} & {t(\Coker(\phi_X))} & 0 \\
	0 & X & {\textbf{Y}} & {\Coker(\phi_X)} & 0
	\arrow[from=1-1, to=1-2]
	\arrow["b", from=1-2, to=1-3]
	\arrow["\id"', from=1-2, to=2-2]
	\arrow["f", from=1-3, to=1-4]
	\arrow[from=1-3, to=2-3]
	\arrow[from=1-4, to=1-5]
	\arrow[from=1-4, to=2-4]
	\arrow[from=2-1, to=2-2]
	\arrow["{\phi_X}", from=2-2, to=2-3]
	\arrow[from=2-3, to=2-4]
	\arrow[from=2-4, to=2-5]
\end{tikzcd}\]
	\end{itemize}
	\begin{proof}
	As in Theorem \ref{3.6} we set $\mathtt{F}_X:=\Ker(\phi_X)$ whenever it is non-zero, and $\mathtt{T}_X:=t(\Coker(\phi_X)$ whenever it is non-zero.
	
	$\Ker(\phi_X)=0$ and $\Coker(\phi_X)\in\mathcal{F}$ is equivalent to the condition that $\phi_X$ is a monomorphism when we consider it as a morphism in the heart $\mathcal{H}_t$. And the latter is equivalent to the condition that the indecomposable injective object $X\in \mathcal{H}_t$ is not injective envelope of a simple object by Theorem \ref{3.4}. This proves $(1)$.
	
	For $(2)$, assume that $\mathtt{F}_X\neq 0$ and set $\bar{X}:=\Imm(\phi_X)$. By Theorem \ref{3.9}, it is enough to show that in the short exact sequence
	\[0\rightarrow\mathtt{F}_X\overset{g}{\longrightarrow} X\overset{a}{\longrightarrow} \bar{X}\rightarrow 0\]
	$g$ is a $\Prod(C)$-envelope. For any $Y\in\Prod(C)$, by applying $\Hom_A(-,Y)$ to this short exact sequence we get the following exact sequence.
	\[0\rightarrow \Hom_A(\bar{X},Y)\rightarrow\Hom_A(X,Y)\rightarrow\Hom_A(\mathtt{F}_X,Y)\rightarrow\Ext^1_A(\bar{X},Y)=0\]
	This proves that $g$ is a $\Prod(C)$-preenvelope. Now let $\alpha:X\rightarrow X$ be a morphism such that $\alpha g=g$. If $g$ is a non-isomorphism, it belongs to $\Rad(X,X)$ by Lemma \ref{3.3}. Thus, by the definition of $\phi_X$ we have $\alpha g=0$. This implies that $g=\alpha g=0$, which is a contradiction.
	
	For proving $(3)$, by Theorem \ref{3.9} we need to show that in the short exact sequence
	\[0\rightarrow X\overset{b}{\longrightarrow}\bar{X}\overset{f}{\longrightarrow} \mathtt{T}_X\rightarrow 0\]
	$f$ is an $\mathcal{F}$-cover. Being a subobject of $\textbf{Y}=\prod_{Y\in\mathcal{Z}_C}Y^{\Rad(X,Y)}$, we have that $\bar{X}\in\mathcal{F}$.
	For an arbitrary object $F\in\mathcal{F}$, by applying $\Hom_A(F,-)$ to this short exact sequence, we obtain the following exact sequence.
	\[0\rightarrow \Hom_A(F,X)\rightarrow\Hom_A(F,\bar{X})\rightarrow\Hom_A(F,\mathtt{T}_X)\rightarrow\Ext^1_A(F,X)=0\]
	So, $f$ is an $\mathcal{F}$-precover. We want to prove that $f$ is an $\mathcal{F}$-cover. Since $\mathcal{F}$ is a covering class, we chose an $\mathcal{F}$-cover $h:F\rightarrow\mathtt{T}_X$. Therefore, using the property of precover, we obtain the commutative diagram
	\[\begin{tikzcd}
	0 & X & {\bar{X}} & {\mathtt{T}_X} & 0 \\
	&& F
	\arrow[from=1-1, to=1-2]
	\arrow["b", from=1-2, to=1-3]
	\arrow["f", from=1-3, to=1-4]
	\arrow["r", curve={height=-6pt}, from=1-3, to=2-3]
	\arrow[from=1-4, to=1-5]
	\arrow["s", curve={height=-6pt}, from=2-3, to=1-3]
	\arrow["h"', from=2-3, to=1-4]
\end{tikzcd}\]
	such that $fs=h$ and $hr=f$. So, $h(rs)=h$, and by minimality of $f$ we have that $rs$ is an isomorphism. Thus, $\bar{X}\cong \Imm(s)\oplus \Ker(r)$.
	On the other hand, $f(\Ker(r))=hr(\Ker(r))=0$. Therefore
	\[X\cong \Ker(f)\cong \Ker(f|_{\Imm(s)})\oplus\Ker(r)\cong \Ker(h)\oplus\Ker(r).\]
	Since $X$ is indecomposable, either $\Ker(h)=0$ or $\Ker(r)=0$.
	The former is impossible, otherwise, being an epimorphism, $h$ would be an isomorphism, which contradict the fact that $F\in\mathcal{F}$ and $\mathtt{T}_X\in\mathcal{T}$.
	Therefore, $\Ker(r)=0$ which means that $r$ is an isomorphism. Thus, $f$ is isomorphic to $h$, so is an $\mathcal{F}$-cover.
	\end{proof}
	\end{theorem}
	
    In what follows we explain how the above results about cotilting torsion pairs, can be adapted to the general case of cosilting torsion pairs. We follow \cite[Section 4]{ALS2}.
    
    Let $C\in\Mod(A)$ be a cosilting module and $(\mathcal{T},\mathcal{F})=({}^{\bot_0}C,\Cogen(C))$ be the torsion pair of finite type cogenerated by $C$. We fix $\bar{A}:=A/\Ann(C)$, where $\Ann(C):=\{a\in A\mid aC=0\}$ is the annihilator of $C$. Then
    \[\Mod(\bar{A})=\{M\in\Mod(A)\mid \Ann(C)M=0\}.\]
    Since $\mathcal{F}\subseteq\Mod(\bar{A})$, one can easily see that $(\bar{\mathcal{T}},\bar{\mathcal{F}}):=(\mathcal{T}\cap\Mod(\bar{A}),\mathcal{F})$ is a torsion pair of finite type in $\Mod(\bar{A})$.
By Proposition \ref{2.8} $C$ is a cotilting module in $\Mod(\bar{A})$ and it cogenerate the torsion pair of finite type $(\bar{\mathcal{T}},\bar{\mathcal{F}})$ in $\Mod(\bar{A})$.
	
	\begin{proposition}$($\cite[Section 4]{ALS2}$)$\label{3.11}
	Let $t=(\ts,\fs)$ be a torsion pair in $\modd(A)$ and $(\mathcal{T},\mathcal{F})$ be associated torsion pair of finite type in $\Mod(A)$.
	\begin{itemize}
	\item[(1)]
	The set of torsion-free almost torsion modules with respect to $(\mathcal{T},\mathcal{F})$ is equal to the set of torsion-free almost torsion modules with respect to $(\bar{\mathcal{T}},\bar{\mathcal{F}})$.
	\item[(2)]
	The set of torsion almost torsion-free modules with respect to $(\mathcal{T},\mathcal{F})$ that admit an epimorphic $\mathcal{F}$-cover is equal to the set of torsion almost torsion-free modules with respect to $(\bar{\mathcal{T}},\bar{\mathcal{F}})$.
	\item[(3)]
	If $T$ is a torsion almost torsion-free module with respect to $(\mathcal{T},\mathcal{F})$ whose $\mathcal{F}$-cover is not epimorphism, then $T$ has a unique maximal submodule $L$, and there is a short exact sequence 
	\[0\rightarrow L\overset{f}{\longrightarrow}T\longrightarrow S\rightarrow 0,\]
	where $f$ is an $\mathcal{F}$-cover in $\Mod(A)$ and the simple $A$-module $S$ belongs to $\mathcal{F}^{\bot_0}$. Also, the injective envelope $I$ of $S$ belongs to $\mathcal{I}_t$ and $0\rightarrow I\in \mathcal{N}_t$ is sent, by the cohomological functor $\HH^0_t$, to the injective envelope of $T[-1]$ in the heart, and this gives a bijection between torsion almost torsion-free modules with respect to $(\mathcal{T},\mathcal{F})$ without epimorphic $\mathcal{F}$-cover and indecomposable injectives in $\mathcal{I}_t$.
	\end{itemize}
    \begin{proof}
    $(1)$ is easy because $\mathcal{F}=\Cogen(C)\subseteq\Mod(\bar{A})$.
    $(2)$ is proved in \cite[Proposition 4.5]{ALS2}. And $(3)$ follows from \cite[Proposition 4.6 and Theorem 4.10]{ALS2}.
    \end{proof}
	\end{proposition}
	By Remark \ref{2.13} the lattice of torsion pairs $\textbf{tors}(A)$ is controlled by maximal rigid sets and also by cosilting pairs. Minimal inclusion of torsion classes correspond to irreducible mutation of maximal rigid sets, or irreducible mutation of cosilting pairs. Let us first recall the definition of mutation. For details and more information the reader is referred to \cite{ALS2}.
	
	\begin{definition}\cite[Definition 3.1]{ALS2}
	Let $\mathcal{L}$ and $\mathcal{R}$ be maximal rigid sets in $\D(A)$. $\mathcal{L}$ and $\mathcal{R}$ are said to be {\it related by mutation} if there is a bijection $\Delta:\mathcal{L}\rightarrow\mathcal{R}$ such that $\Delta(\mu)=\mu$ for every $\mu\in\mathcal{L}\cap\mathcal{R}$ and, for every $\lambda\in\mathcal{L}\backslash\mathcal{R}$, there exists a triangle 
	\[\begin{tikzcd}
	{\Delta(\lambda)} & {\epsilon_{\lambda}} & \lambda & {\Delta(\lambda)[-1]}
	\arrow["{\Psi_{\lambda}}", from=1-1, to=1-2]
	\arrow["{\Phi_{\lambda}}", from=1-2, to=1-3]
	\arrow[from=1-3, to=1-4]
\end{tikzcd}\]
where $\Phi_{\lambda}$ is a $\Prod(\mathcal{L}\cap\mathcal{R})$-cover of $\lambda$ and $\Psi_{\lambda}$ is a $\Prod(\mathcal{L}\cap\mathcal{R})$-envelope of $\Delta(\lambda)$.
Then, $\mathcal{L}$ is called a {\it left mutation of} $\mathcal{R}$ at $\mathcal{R}\backslash\mathcal{L}$ and $\mathcal{R}$ is called a {\it right mutation of} $\mathcal{L}$ at $\mathcal{L}\backslash\mathcal{R}$. The mutation is called {\it irreducible} if $\vert\mathcal{L}\backslash\mathcal{R} \vert=1$.

Let $\mathcal{M}$ be a maximal rigid set. An element $\mu\in\mathcal{M}$ is called {\it left mutable in} $\mathcal{M}$ if there exists a maximal rigid set $\mathcal{N}=(\mathcal{M}\backslash{\mu})\cup\{\nu\}$ such that $\mathcal{N}$ is a left mutation of $\mathcal{M}$. Being {\it right mutable in} $\mathcal{M}$ is defined similarly.

Let $(\mathcal{Z},\mathcal{I})$ be a cosilting pair and $\mathcal{N}$ be the maximal rigid set correspond to $(\mathcal{Z},\mathcal{I})$ under the bijections of Remark \ref{2.13}. $X\in\mathcal{M}$ (resp. $I\in\mathcal{I}$) is called {\it left mutable in} $(\mathcal{Z},\mathcal{I})$ if $\mu_X$ (resp. $I[-1]$) is left mutable in $\mathcal{M}$. Being {\it right mutable in} $(\mathcal{Z},\mathcal{I})$ is defined similarly.
\end{definition}
	
\begin{proposition}$($\cite[Corollary 3.10]{ALS2}$)$\label{3.13}
Let $t=(\ts,\fs)$ be a torsion pair in $\modd(A)$.
\begin{itemize}
\item[(1)]
An element $\mu$ in the maximal rigid set $\mathcal{N}_t$ (and the corresponding element in the cosilting pair $(\mathcal{Z}_t,\mathcal{I}_t)$) is left mutable if and only if $\HH^0_t(\mu)$ is injective envelope of a torsion-free simple in $\mathcal{H}_t$ (which is always finite dimensional).
\item[(2)]
An element $\mu$ in the maximal rigid set $\mathcal{N}_t$ (and the corresponding element in the cosilting pair $(\mathcal{Z}_t,\mathcal{I}_t)$) is right mutable if and only if $\HH^0_t(\mu)$ is injective envelope of a finite dimensional torsion simple in $\mathcal{H}_t$.
\end{itemize}
\end{proposition}

\begin{proposition}$($\cite[Corollary 4.11]{ALS2}$)$\label{3.14}
   The injective envelops of simple objects $S$ in $\mathcal{H}_t$ are exactly the following.
   \begin{itemize}
       \item[(1)]
       $\HH^0_t(\mu_X)$ for a critical module $X\in\mathcal{Z}_t$. These are precisely the injective envelopes of $S\cong F$ with $F$ torsion-free almost torsion with respect to $(\mathcal{T},\mathcal{F})$.
       \item[(2)]
       $\HH^0_t(\mu_X)$ for a special module $X\in\mathcal{Z}_t$. These are precisely the injective envelopes of $S\cong T[-1]$ with $T$ torsion almost torsion-free with respect to $(\mathcal{T},\mathcal{F})$ with an epimorphic $\mathcal{F}$-cover.
       \item[(3)]
       $\HH^0_t(0\rightarrow I)$ for $I\in\mathcal{I}_t$. These are precisely the injective envelopes of $S\cong T[-1]$ with $T$ torsion almost torsion-free with respect to $(\mathcal{T},\mathcal{F})$ without an epimorphic $\mathcal{F}$-cover.
   \end{itemize}
\end{proposition}
	
Now we can prove the following last result, which gives some necessary and sufficient conditions for an element of a cosilting pair (and a maximal rigid set) to be left mutable or right mutable. Here we use the Notation \ref{2.17}.

\begin{theorem}\label{3.15}
Let $(\mathcal{Z}_t,\mathcal{I}_t)$ be a cosilting pair associated to a torsion pair $t\in\textbf{tors}(A)$. For $X\in\mathcal{Z}_t$ consider the canonical morphism
    \begin{equation*}
		\phi_X:X\longrightarrow \prod_{Y\in\mathcal{Z}_t}Y^{\Rad(X,Y)}.
		\end{equation*}
   	\begin{itemize}
\item[(1)]
Every $I\in\mathcal{I}_t$ is left mutable in $(\mathcal{Z}_t,\mathcal{I}_t)$.
\item[(2)]
$X\in\mathcal{Z}_t$ is left mutable if and only if $\mathtt{T}_X:=t(\Coker(\phi_X))\neq 0$, or equivalently                   $\Coker(\phi_X)\notin \mathcal{F}$.
\item[(3)]
$X\in\mathcal{Z}_t$ is right mutable if and only if $\mathtt{F}_X:=\Ker(\phi_X)$ is a finite dimensional module.
\end{itemize}
\begin{proof}
The proof follow from Theorem \ref{3.10} and Propositions \ref{3.13} and \ref{3.14}. Note that in $(2)$ if $\mathtt{T}_X\neq 0$, it is automatically finite dimensional by Proposition \ref{2.21}.
\end{proof}
\end{theorem}
	
	\section*{acknowledgements}
Ramin Ebrahimi is supported by Zhejiang Normal University.
Rasool Hafezi is supported by the National Natural Science of China (Grant No. 12571042).
Jiaqun Wei is supported by the National Natural Science Foundation of China (Grant Nos. 12571042, 12271249) and the Natural Science Foundation of Zhejiang Province (Grant No. LZ25A010002).

\end{document}